\documentclass[12pt,reqno]{amsart}
\usepackage{eurosym}
\usepackage{amsfonts}
\usepackage{amsmath}
\usepackage{amssymb,latexsym}
\usepackage{enumerate}
\usepackage[bookmarksnumbered, colorlinks, plainpages]{hyperref}
\usepackage{amsbsy}
\usepackage{amscd}
\usepackage{epsfig}
\usepackage{graphicx}
\usepackage{epstopdf}
\usepackage{caption}
\usepackage{subcaption}

\newtheorem{theorem}{Theorem}[section]

\newtheorem{definition}{Definition}[section]

\newtheorem{example}{Example}[section]
\newtheorem{remark}{Remark}[section]
\numberwithin{equation}{section}

\begin{document}
\date{{\scriptsize Received: July 2021, Accepted: .}}
\title[$\varphi $-fixed points]{$\varphi $-fixed points of self-mappings on
metric spaces with a geometric viewpoint}
\author[N. \"{O}ZG\"{U}R]{Nihal \"{O}ZG\"{U}R}
\address{Bal\i kesir University, Department of Mathematics, 10145 Bal\i
kesir, Turkey}
\email{nihal@balikesir.edu.tr}
\author[N. TA\c{S}]{Nihal TA\c{S}}
\address{Bal\i kesir University, Department of Mathematics, 10145 Bal\i
kesir, Turkey}
\email{nihaltas@balikesir.edu.tr}
\maketitle

\begin{abstract}
A recent open problem was stated on the geometric properties of $\varphi $%
-fixed points of self-mappings of a metric space in the non-unique fixed
point cases. In this paper, we deal with the solutions of this open problem
and present some solutions via the help of appropriate auxiliary numbers and
geometric conditions. We see that a zero of a given function $\varphi $ can
produce a fixed circle (resp. fixed disc) contained in the fixed point set
of a self-mapping $T$ on a metric space. Moreover, this circle (resp. fixed
disc) is also contained in the set of zeros of the function $\varphi $.


\end{abstract}

\section{\textbf{Introduction and Motivation}}

\label{intro} Let $T$ $:X\rightarrow X$ be a self-mapping on a metric space $%
(X,d)$. We denote the fixed point set of $T$ by $Fix(T)$, that is, we have
\begin{equation*}
Fix(T)=\left\{ x\in X:Tx=x\right\} \text{.}
\end{equation*}%
In this paper, mainly, we study on the geometric properties of the set $%
Fix(T)$ in the cases where this set is not a singleton. Recently, geometric
properties of non-unique fixed points have been extensively studied by
various aspects, for example, in the context of the fixed-circle problem,
fixed-disc problem and so on (see \cite{Mlaiki arxiv, Ozgur-malaysian,
Ozgur-Aip, Ozgur-chapter, Ozgur-fixed disc, Ozgur-simulation, Ozgur 2021,
Tas 2020} and the references therein).

Auxiliary numbers are useful tools for the fixed-point studies. In a recent
study, new solutions to the Rhoades' well-known problem related to
discontinuity at the fixed point were investigated using the auxiliary
number $M(x,y)$ defined by%
\begin{equation}
M(x,y)=\max \left\{ d(x,y),d(x,Tx),d(y,Ty),\left[ \frac{d(x,Ty)+d(y,Tx)}{%
1+d(x,Tx)+d(y,Ty)}\right] d(x,y)\right\} \text{,}  \label{M(x,y)}
\end{equation}%
for all $x,y\in X$ (see \cite{Ozgur 2021} for more details). Furthermore,
the number $M(x,y)$ and its modified versions were used to give some
fixed-point results, a common fixed-point theorem and an application to the
fixed-circle problem. It was pointed out that new results on the geometric
properties of the $\varphi $-fixed points of a self-mapping can be
investigated via the help of the number $M(x,y)$ (or a modified version of
it). The notion of a $\varphi $-fixed point was introduced in \cite{Jleli}.
An element $x\in X$ is said to be a $\varphi $-fixed point of the
self-mapping $T:X\rightarrow X$, where $\varphi :X\rightarrow \left[
0,\infty \right) $ is a given function, if $x$ is a fixed point of $T$ and $%
\varphi \left( x\right) =0$ \cite{Jleli}. Recently, several existence
results of $\varphi $-fixed points for various classes of operators have
been established (see for example \cite{Agarwal, Gopal, Imdad, Jleli,
Karapinar, Kumrod, Kumrod 2019, Vetro}). Following \cite{Jleli}, we denote
the set of all zeros of the function $\varphi $ by $Z_{\varphi }$, that is,
we have%
\begin{equation*}
Z_{\varphi }=\left\{ x\in X:\varphi \left( x\right) =0.\right\}
\end{equation*}%
Considering the non-unique fixed point cases, in \cite{Ozgur 2021}, the
notions of a $\varphi $-fixed circle and of a $\varphi $-fixed disc were
defined as follows:

\begin{definition}
\cite{Ozgur 2021} Let $(X,d)$ be a metric space, $T$ be a self-mapping of $X$
and $\varphi :X\rightarrow \left[ 0,\infty \right) $ be a given function.

$1)$ A circle $C_{x_{0},r}=\left\{ x\in X:d\left( x,x_{0}\right) =r\right\} $
in $X$ is said to be a $\varphi $-fixed circle of $T$ if and only if $%
C_{x_{0},r}\subseteq Fix(T)\cap Z_{\varphi }$.

$2)$ A disc $D_{x_{0},r}=\left\{ x\in X:d\left( x,x_{0}\right) \leq
r\right\} $ in $X$ is said to be a $\varphi $-fixed disc of $T$ if and only
if $D_{x_{0},r}\subseteq Fix(T)\cap Z_{\varphi }$.
\end{definition}

\begin{example}
\label{exm11} \cite{Ozgur 2021} Let $X=%
%TCIMACRO{\U{211d} }%
%BeginExpansion
\mathbb{R}
%EndExpansion
$ be the usual metric space. Let us consider the self-mapping $T:%
%TCIMACRO{\U{211d} }%
%BeginExpansion
\mathbb{R}
%EndExpansion
\rightarrow
%TCIMACRO{\U{211d} }%
%BeginExpansion
\mathbb{R}
%EndExpansion
$ defined by%
\begin{equation*}
Tx=x^{4}-5x^{2}+x+4
\end{equation*}%
and the function $\varphi :%
%TCIMACRO{\U{211d} }%
%BeginExpansion
\mathbb{R}
%EndExpansion
\rightarrow \left[ 0,\infty \right) $ defined by%
\begin{equation}
\varphi \left( x\right) =\left\vert x-1\right\vert +\left\vert
x+1\right\vert -2\text{.}  \label{phi function}
\end{equation}%
Then, we have
\begin{equation*}
Fix\left( T\right) =\left\{ x\in
%TCIMACRO{\U{211d} }%
%BeginExpansion
\mathbb{R}
%EndExpansion
:Tx=x\right\} =\left\{ -2,-1,1,2\right\}
\end{equation*}%
and%
\begin{equation*}
Z_{\varphi }=\left\{ x\in
%TCIMACRO{\U{211d} }%
%BeginExpansion
\mathbb{R}
%EndExpansion
:\varphi \left( x\right) =0\right\} =\left[ -1,1\right] \text{. }
\end{equation*}%
Clearly, we obtain%
\begin{equation*}
Fix\left( T\right) \cap Z_{\varphi }=C_{0,1}=\left\{ -1,1\right\} \text{,}
\end{equation*}%
that is, the circle $C_{0,1}$ is the $\varphi $-fixed circle of $T$.

On the other hand, if we consider the self-mapping $S:%
%TCIMACRO{\U{211d} }%
%BeginExpansion
\mathbb{R}
%EndExpansion
\rightarrow
%TCIMACRO{\U{211d} }%
%BeginExpansion
\mathbb{R}
%EndExpansion
$ defined by $Sx=x^{3}$ together with the function $\varphi \left( x\right) $
defined in $($\ref{phi function}$)$, we have $Fix\left( S\right) =\left\{
-1,0,1\right\} $ and hence
\begin{equation*}
C_{0,1}\subset Fix\left( S\right) \cap Z_{\varphi }\text{.}
\end{equation*}%
Consequently, the circle $C_{0,1}$ is the $\varphi $-fixed circle of $S$.
\end{example}

\begin{example}
\label{exm12} \cite{Ozgur 2021} Let us consider the function $\varphi \left(
x\right) $ defined in $($\ref{phi function}$)$ together with the
self-mapping $T:%
%TCIMACRO{\U{211d} }%
%BeginExpansion
\mathbb{R}
%EndExpansion
\rightarrow
%TCIMACRO{\U{211d} }%
%BeginExpansion
\mathbb{R}
%EndExpansion
$ defined by%
\begin{equation*}
Tx=\left\{
\begin{array}{ccc}
x & ; & \left\vert x\right\vert \leq 1 \\
x+2 & ; & \left\vert x\right\vert >1%
\end{array}%
\right. .
\end{equation*}%
Since we have $D_{0,1}=Fix\left( T\right) \cap Z_{\varphi }$, the disc $%
D_{0,1}=\left[ -1,1\right] $ is a $\varphi $-fixed disc of $T$. Notice that
the disc $D_{0,\frac{1}{2}}=\left[ -\frac{1}{2},\frac{1}{2}\right] $ is
another $\varphi $-fixed disc of $T$.
\end{example}

Considering the above examples, the investigation of the existence and
uniqueness of $\varphi $-fixed circles (resp. $\varphi $-fixed discs) seems
to be an interesting problem for various classes of self-mappings. In \cite%
{Ozgur 2021}, it was proposed the usage of the number $M(x,y)$ or a modified
version of it for the studies of this direction.

In this paper, we obtain some solutions to this mentioned problem via the
help of appropriate auxiliary numbers and geometric conditions. We see that
any zero of a given function $\varphi :X\rightarrow \left[ 0,\infty \right) $
can produce a fixed circle (resp. fixed disc) contained in the set $%
Fix\left( T\right) \cap Z_{\varphi }$ for a self-mapping $T$ on a metric
space.

\section{\textbf{$\protect\varphi $-Fixed Circle and $\protect\varphi $%
-Fixed Disc Results}}

\label{sec:1} In this section, using the number $M(x,y)$ defined in (\ref%
{M(x,y)}), the numbers $\rho $ and $\mu $ defined by%
\begin{equation}
\rho :=\inf \left\{ d\left( Tx,x\right) :x\in X,x\neq Tx\right\}
\label{the number rho}
\end{equation}%
and
\begin{equation}
\mu :=\inf \left\{ \sqrt{d\left( Tx,x\right) }:x\in X,x\neq Tx\right\} ,
\label{the number mu}
\end{equation}%
we give several $\varphi $-fixed circle (resp. $\varphi $-fixed disc)
results using various geometric conditions and techniques.

\subsection{\textbf{$\protect\varphi $-Fixed Circle (resp. $\protect\varphi $%
-Fixed Disc) Results via Type 1 $\protect\varphi _{x_{0}}$-Contractions}}

First, we define a new contraction type on a metric space.

\begin{definition}
\label{def:21} Let $(X,d)$ be a metric space, $T$ be a self-mapping of $X$
and $\varphi :X\rightarrow \left[ 0,\infty \right) $ be a given function. If
there exists a point $x_{0}\in X$ such that
\begin{equation}
d\left( Tx,x\right) >0\Rightarrow \max \left\{ d\left( x,Tx\right) ,\varphi
\left( Tx\right) ,\varphi \left( x\right) \right\} \leq k\max \left\{
d\left( x,x_{0}\right) ,\varphi \left( x\right) ,\varphi \left( x_{0}\right)
\right\} ,  \label{eq:21}
\end{equation}%
for all $x\in X$ and some $k\in \left( 0,1\right) $, then $T$ is called a
type $1$ $\varphi _{x_{0}}$-contraction.
\end{definition}

In the following theorems, we see that the number $\rho $ defined in (\ref%
{the number rho}) and the point $x_{0}$ produce a $\varphi $-fixed circle
(resp. $\varphi $-fixed disc) under a geometric condition.

\begin{theorem}
\label{thm:21} Let $(X,d)$ be a metric space, the number $\rho $ be defined
as in $($\ref{the number rho}$)$ and $T:X\rightarrow X$ be a type $1$ $%
\varphi _{x_{0}}$-contraction with the point $x_{0}\in X$ and the given
function\ $\varphi :X\rightarrow \left[ 0,\infty \right) $. If $x_{0}\in
Z_{\varphi }$ and%
\begin{equation}
\varphi \left( x\right) \leq d\left( Tx,x\right)   \label{eq:22}
\end{equation}%
for all $x\in C_{x_{0},\rho }$, then the circle $C_{x_{0},\rho }$ is a $%
\varphi $-fixed circle of $T$.
\end{theorem}

\begin{proof}
At first, we show that $x_{0}\in Fix\left( T\right) $. Conversely, assume
that $x_{0}\neq Tx_{0}$. Then we have $d\left( Tx_{0},x_{0}\right) >0$ and
using the inequality (\ref{eq:21}) together with the hypothesis $x_{0}\in
Z_{\varphi }$, we find%
\begin{equation*}
\max \left\{ d\left( x_{0},Tx_{0}\right) ,\varphi \left( Tx_{0}\right)
,\varphi \left( x_{0}\right) \right\} \leq k\max \left\{ d\left(
x_{0},x_{0}\right) ,\varphi \left( x_{0}\right) ,\varphi \left( x_{0}\right)
\right\} =0,
\end{equation*}%
and hence
\begin{equation*}
\max \left\{ d\left( x_{0},Tx_{0}\right) ,\varphi \left( Tx_{0}\right)
\right\} =0.
\end{equation*}%
This implies $d\left( x_{0},Tx_{0}\right) =0$, which is a contradiction with
our assumption. Therefore, it should be $Tx_{0}=x_{0}$, that is, $x_{0}\in
Fix\left( T\right) \cap Z_{\varphi }$.

Now we have two cases.

\textbf{Case 1.} If $\rho =0$, then clearly $C_{x_{0},\rho }=\left\{
x_{0}\right\} \subset Fix\left( T\right) \cap Z_{\varphi }$ and hence, the
circle $C_{x_{0},\rho }$ is a $\varphi $-fixed circle of $T$.

\textbf{Case 2.} Let $\rho >0$. For any $x\in C_{x_{0},\rho }$ with $Tx\neq
x $, we have
\begin{equation*}
\max \left\{ d\left( x,Tx\right) ,\varphi \left( Tx\right) ,\varphi \left(
x\right) \right\} \leq k\max \left\{ d\left( x,x_{0}\right) ,\varphi \left(
x\right) \right\} .
\end{equation*}%
If $\max \left\{ d\left( x,x_{0}\right) ,\varphi \left( x\right) \right\}
=d\left( x,x_{0}\right) =\rho $, then by the definition of the number $\rho $%
, we get
\begin{equation*}
\max \left\{ d\left( x,Tx\right) ,\varphi \left( Tx\right) ,\varphi \left(
x\right) \right\} \leq k\rho \leq kd\left( x,Tx\right)
\end{equation*}%
and so $d\left( x,Tx\right) \leq kd\left( x,Tx\right) $, a contradiction by
the hypothesis $k\in \left( 0,1\right) $.

If $\max \left\{ d\left( x,x_{0}\right) ,\varphi \left( x\right) \right\}
=\varphi \left( x\right) $, we obtain%
\begin{equation*}
\max \left\{ d\left( x,Tx\right) ,\varphi \left( Tx\right) ,\varphi \left(
x\right) \right\} \leq k\varphi \left( x\right)
\end{equation*}%
and hence $\varphi \left( x\right) \leq k\varphi \left( x\right) $, a
contradiction.

Then, it should be $Tx=x$, that is, $x\in Fix\left( T\right) $. By (\ref%
{eq:22}), we have $\varphi \left( x\right) =0$ for all $x\in C_{x_{0},\rho }$%
. This implies $x\in Fix\left( T\right) \cap Z_{\varphi }$ for all $x\in
C_{x_{0},\rho }$. Consequently, we find $C_{x_{0},\rho }\subset Fix\left(
T\right) \cap Z_{\varphi }$ and hence, the circle $C_{x_{0},\rho }$ is a $%
\varphi $-fixed circle of $T$.
\end{proof}

\begin{theorem}
\label{thm:22} Let $(X,d)$ be a metric space, the number $\rho $ be defined
as in $($\ref{the number rho}$)$, $T:X\rightarrow X$ be a type $1$ $\varphi
_{x_{0}}$-contraction with the point $x_{0}\in X$ and the given function\ $%
\varphi :X\rightarrow \left[ 0,\infty \right) $. If $x_{0}\in Z_{\varphi }$
and the inequality
\begin{equation*}
\varphi \left( x\right) \leq d\left( Tx,x\right)
\end{equation*}%
is satisfied for all $x\in D_{x_{0},\rho }$, then the disc $D_{x_{0},\rho }$
is a $\varphi $-fixed disc of $T$.
\end{theorem}

\begin{proof}
The proof follows from the proof of Theorem \ref{thm:21}.
\end{proof}

We give an illustrative example of Theorem \ref{thm:21} and Theorem \ref%
{thm:22}.

\begin{example}
\label{exm1} Let $\left(
%TCIMACRO{\U{211d} }%
%BeginExpansion
\mathbb{R}
%EndExpansion
,d\right) $ be the usual metric space and consider the self-mapping $%
T:X\rightarrow X$ defined by%
\begin{equation*}
Tx=\left\{
\begin{array}{ccc}
\frac{x}{2} & , & x>2 \\
x & ; & x\leq 2%
\end{array}%
\right.
\end{equation*}%
and the function $\varphi :$ $%
%TCIMACRO{\U{211d} }%
%BeginExpansion
\mathbb{R}
%EndExpansion
\rightarrow \left[ 0,\infty \right) $ defined by
\begin{equation*}
\varphi \left( x\right) =\left\{
\begin{array}{ccc}
\frac{x}{4} & ; & x>0 \\
0 & ; & x\leq 0%
\end{array}%
\right. .
\end{equation*}%
We show that $T$ is a type $1$ $\varphi _{x_{0}}$-contraction with the point
$x_{0}=-1$ and $k=\frac{1}{2}$. Indeed, we have
\begin{equation*}
\max \left\{ \left\vert x-\frac{x}{2}\right\vert ,\frac{x}{8},\frac{x}{4}%
\right\} =\frac{x}{2}\leq \frac{1}{2}\max \left\{ \left\vert x+1\right\vert ,%
\frac{x}{4},0\right\} =\frac{x+1}{2}=\frac{x}{2}+\frac{1}{2}
\end{equation*}%
for all $x>2$. We find%
\begin{eqnarray*}
\rho  &=&\inf \left\{ d\left( Tx,x\right) :x\in X,x\neq Tx\right\}  \\
&=&\inf \left\{ \left\vert x-\frac{x}{2}\right\vert =\frac{x}{2}:x>2\right\}
\\
&=&1
\end{eqnarray*}%
and the conditions of Theorem \ref{thm:21} are satisfied by $T$. Observe
that $Fix\left( T\right) \cap Z_{\varphi }=\left( -\infty ,0\right] $ and we
get
\begin{equation*}
C_{-1,1}=\left\{ -2,0\right\} \subset Fix\left( T\right) \cap Z_{\varphi }.
\end{equation*}%
Hence, the circle $C_{-1,1}$ is a $\varphi $-fixed circle of $T$.

Clearly, $T$ also satisfies the conditions of Theorem \ref{thm:22} and the
disc $D_{-1,1}$ is a $\varphi $-fixed disc of $T$.

Notice that the condition $\varphi \left( x\right) \leq d\left( Tx,x\right) $
is curicial for both Theorem \ref{thm:21} and Theorem \ref{thm:22}. For any $%
x_{0}\in \left( -1,0\right] $, this condition is not satisfied for the
points $x=1+x_{0}\in C_{x_{0},1}$ and $x\in D_{x_{0},1}$ with $x>0$. In
fact, it is easy to check that $T$ is also a type $1$ $\varphi _{x_{0}}$%
-contraction with each of the points $x_{0}\in \left( -1,0\right] $.
\end{example}

Now, we define a new contraction type using the auxiliary number $M(x,y)$
defined in (\ref{M(x,y)}).

\begin{definition}
\label{def:22} Let $(X,d)$ be a metric space, $T$ be a self-mapping of $X$
and $\varphi :X\rightarrow \left[ 0,\infty \right) $ be a given function. If
there exists a point $x_{0}\in X$ such that
\begin{equation}
d\left( Tx,x\right) >0\Rightarrow \max \left\{ d\left( x,Tx\right) ,\varphi
\left( Tx\right) ,\varphi \left( x\right) \right\} \leq k\max \left\{
M\left( x,x_{0}\right) ,\varphi \left( x\right) ,\varphi \left( x_{0}\right)
\right\} ,  \label{eq:23}
\end{equation}%
for all $x\in X$ and some $k\in \left( 0,\frac{1}{2}\right) $, then $T$ is
called a generalized type $1$ $\varphi _{x_{0}}$-contraction.
\end{definition}

We prove the following $\varphi $-fixed circle theorem using the number $\mu
$ defined in (\ref{the number mu}). We note that
\begin{eqnarray*}
M\left( x_{0},x_{0}\right) &=&\max \left\{
d(x_{0},x_{0}),d(x_{0},Tx_{0}),d(x_{0},Tx_{0}),\left[ \frac{%
d(x_{0},Tx_{0})+d(x_{0},Tx_{0})}{1+d(x_{0},Tx_{0})+d(x_{0},Tx_{0})}\right]
d(x_{0},x_{0})\right\} \\
&=&d(x_{0},Tx_{0}).
\end{eqnarray*}

\begin{theorem}
\label{thm:23} Let $(X,d)$ be a metric space, the number $\mu $ be defined
as in $($\ref{the number mu}$)$ and $T:X\rightarrow X$ be a generalized type
$1$ $\varphi _{x_{0}}$-contraction with the point $x_{0}\in X$ and the given
function\ $\varphi :X\rightarrow \left[ 0,\infty \right) $. If $x_{0}\in
Z_{\varphi }$ and the inequalities%
\begin{equation}
d\left( Tx,x_{0}\right) \leq \mu \text{,}  \label{eq:24}
\end{equation}%
\begin{equation*}
\varphi \left( x\right) \leq d\left( Tx,x\right)
\end{equation*}%
hold for all $x\in C_{x_{0},\mu }$, then the circle $C_{x_{0},\mu }$ is a $%
\varphi $-fixed circle of $T$.
\end{theorem}

\begin{proof}
Suppose that $x_{0}\neq Tx_{0}$. Then we have $d\left( Tx_{0},x_{0}\right)
>0 $ and using the inequality (\ref{eq:23}) and the hypothesis $x_{0}\in
Z_{\varphi }$, we find%
\begin{eqnarray*}
\max \left\{ d\left( x_{0},Tx_{0}\right) ,\varphi \left( Tx_{0}\right)
,\varphi \left( x_{0}\right) \right\} &\leq &k\max \left\{ M\left(
x_{0},x_{0}\right) ,\varphi \left( x_{0}\right) ,\varphi \left( x_{0}\right)
\right\} \\
&\leq &kM\left( x_{0},x_{0}\right) \\
&=&kd(x_{0},Tx_{0})
\end{eqnarray*}%
and hence
\begin{equation*}
\max \left\{ d\left( x_{0},Tx_{0}\right) ,\varphi \left( Tx_{0}\right)
\right\} \leq kd(x_{0},Tx_{0}).
\end{equation*}%
This implies $d\left( x_{0},Tx_{0}\right) \leq kd(x_{0},Tx_{0})$, which is a
contradiction since $k\in \left( 0,\frac{1}{2}\right) $. It should be $%
Tx_{0}=x_{0}$, that is, we get $x_{0}\in Fix\left( T\right) $. Therefore, we
have $x_{0}\in Fix\left( T\right) \cap Z_{\varphi }$.

If $\mu =0$, then clearly $C_{x_{0},\mu }=\left\{ x_{0}\right\} \subset
Fix\left( T\right) \cap Z_{\varphi }$ and hence, the circle $C_{x_{0},\mu }$
is a $\varphi $-fixed circle of $T$.

Now, let $\mu >0$. For any $x\in C_{x_{0},\mu }$ with $Tx\neq x$, we have
\begin{equation*}
\max \left\{ d\left( x,Tx\right) ,\varphi \left( Tx\right) ,\varphi \left(
x\right) \right\} \leq k\max \left\{ M\left( x,x_{0}\right) ,\varphi \left(
x\right) \right\} .
\end{equation*}%
If $\max \left\{ M\left( x,x_{0}\right) ,\varphi \left( x\right) \right\}
=\varphi \left( x\right) $, we obtain%
\begin{equation*}
\max \left\{ d\left( x,Tx\right) ,\varphi \left( Tx\right) ,\varphi \left(
x\right) \right\} \leq k\varphi \left( x\right)
\end{equation*}%
and hence $\varphi \left( x\right) \leq k\varphi \left( x\right) $, a
contradiction by the hypothesis $k\in \left( 0,\frac{1}{2}\right) $.

Let $\max \left\{ M\left( x,x_{0}\right) ,\varphi \left( x\right) \right\}
=M\left( x,x_{0}\right) $. We have
\begin{eqnarray*}
M(x,x_{0}) &=&\max \left\{ d(x,x_{0}),d(x,Tx),d(x_{0},Tx_{0}),\left[ \frac{%
d(x,Tx_{0})+d(x_{0},Tx)}{1+d(x,Tx)+d(x_{0},Tx_{0})}\right] d(x,x_{0})\right\}
\\
&\leq &\max \left\{ \mu ,d(x,Tx),0,\left[ \frac{\mu +\mu }{1+d(x,Tx)}\right]
\mu \right\} \\
&=&\max \left\{ \mu ,d(x,Tx),0,\frac{2\mu ^{2}}{1+d(x,Tx)}\right\} .
\end{eqnarray*}%
If $M(x,x_{0})=d(x,Tx)$ then we get
\begin{equation*}
\max \left\{ d\left( x,Tx\right) ,\varphi \left( Tx\right) ,\varphi \left(
x\right) \right\} \leq kd\left( x,Tx\right)
\end{equation*}%
and so $d\left( x,Tx\right) \leq kd\left( x,Tx\right) $, a contradiction by
the hypothesis $k\in \left( 0,\frac{1}{2}\right) $. If $M(x,x_{0})=\frac{%
2\mu ^{2}}{1+d(x,Tx)}$ then by the definition of the number $\mu $ we have%
\begin{eqnarray*}
\max \left\{ d\left( x,Tx\right) ,\varphi \left( Tx\right) ,\varphi \left(
x\right) \right\} &\leq &k\frac{2\mu ^{2}}{1+d(x,Tx)} \\
&\leq &\frac{2k\left( \sqrt{d(x,Tx)}\right) ^{2}}{1+d(x,Tx)} \\
&<&2kd(x,Tx)
\end{eqnarray*}%
and hence $d\left( x,Tx\right) <2k(x,Tx)$, a contradiction by the hypothesis
$k\in \left( 0,\frac{1}{2}\right) $. Then, it should be $Tx=x$, that is, $%
x\in Fix\left( T\right) $ in all of the above cases. By the hypothesis $%
\varphi \left( x\right) \leq d\left( Tx,x\right) $, we have $\varphi \left(
x\right) =0$ for all $x\in C_{x_{0},\mu }$. This implies $x\in Fix\left(
T\right) \cap Z_{\varphi }$ for all $x\in C_{x_{0},\mu }$. Consequently, we
find $C_{x_{0},\mu }\subset Fix\left( T\right) \cap Z_{\varphi }$ and hence,
the circle $C_{x_{0},\mu }$ is a $\varphi $-fixed circle of $T$.
\end{proof}

\begin{theorem}
\label{thm:24} Let $(X,d)$ be a metric space, the number $\mu $ be defined
as in $($\ref{the number mu}$)$ and $T:X\rightarrow X$ be a type $1$
generalized $\varphi _{x_{0}}$-contraction with the point $x_{0}\in X$ and
the given function\ $\varphi :X\rightarrow \left[ 0,\infty \right) $. If $%
x_{0}\in Z_{\varphi }$ and the inequalities%
\begin{equation*}
d\left( Tx,x_{0}\right) \leq \mu \text{,}
\end{equation*}%
\begin{equation*}
\varphi \left( x\right) \leq d\left( Tx,x\right)
\end{equation*}%
hold for all $x\in D_{x_{0},\mu }$, then the disc $D_{x_{0},\mu }$ is a $%
\varphi $-fixed disc of $T$.
\end{theorem}

\begin{proof}
The proof is similar to the proof of Theorem \ref{thm:23}.
\end{proof}

Now, we give two illustrative examples.

\begin{example}
\label{exm2} Let $\left(
%TCIMACRO{\U{211d} }%
%BeginExpansion
\mathbb{R}
%EndExpansion
,d\right) $ be the usual metric space and consider the self-mapping $%
T:X\rightarrow X$ defined by%
\begin{equation*}
Tx=\left\{
\begin{array}{ccc}
x & ; & x\geq -2 \\
x-1 & ; & x<-2%
\end{array}%
\right.
\end{equation*}%
and the function $\varphi :$ $%
%TCIMACRO{\U{211d} }%
%BeginExpansion
\mathbb{R}
%EndExpansion
\rightarrow \left[ 0,\infty \right) $ defined by
\begin{equation*}
\varphi \left( x\right) =\left\vert x\right\vert -x=\left\{
\begin{array}{ccc}
0 & ; & x\geq -1 \\
-2x & ; & x<-1%
\end{array}%
\right. .
\end{equation*}%
We have%
\begin{eqnarray*}
\mu &=&\inf \left\{ \sqrt{d\left( Tx,x\right) }:x\in X,x\neq Tx\right\} \\
&=&\inf \left\{ \sqrt{\left\vert x-1-x\right\vert }:x<-2\right\} \\
&=&1
\end{eqnarray*}%
We show that $T$ is not a type $1$ $\varphi _{x_{0}}$-contraction with the
point $x_{0}=0$ and any $k\in \left( 0,1\right) $. Indeed, we have
\begin{equation*}
\max \left\{ d\left( x,Tx\right) ,\varphi \left( Tx\right) ,\varphi \left(
x\right) \right\} =\max \left\{ \left\vert x-1-x\right\vert ,-2x,-2x\right\}
=-2x
\end{equation*}%
and%
\begin{equation*}
\max \left\{ d\left( x,0\right) ,\varphi \left( Tx\right) ,\varphi \left(
0\right) \right\} =\max \left\{ \left\vert x-0\right\vert ,-2x,0\right\}
=-2x,
\end{equation*}%
for all $x<-2$.

On the other hand, we have%
\begin{equation*}
M(x,0)=\max \left\{ \left\vert x\right\vert ,1,0,\left[ \frac{\left\vert
x\right\vert +\left\vert x-1\right\vert }{1+1}\right] \left\vert
x\right\vert \right\} =\frac{\left\vert x\right\vert +\left\vert
x-1\right\vert }{2}\left\vert x\right\vert ,
\end{equation*}%
for all $x<-2$. If we choose $k=\frac{1}{4}$, then $($\ref{eq:23}$)$ is
satisfied and so $T$ is a generalized type $1$ $\varphi _{x_{0}}$%
-contraction with the point $x_{0}=0$.

Observe that
\begin{equation*}
Fix\left( T\right) \cap Z_{\varphi }=\left[ -2,\infty \right) \cap \left[
-1,\infty \right) =\left[ -1,\infty \right)
\end{equation*}%
and we get
\begin{equation*}
C_{0,1}=\left\{ -1,1\right\} \subset Fix\left( T\right) \cap Z_{\varphi }.
\end{equation*}%
Hence, the circle $C_{0,1}$ is a $\varphi $-fixed circle of $T$.

Clearly, $T$ also satisfies the conditions of Theorem \ref{thm:24} and the
disc $D_{0,1}$ is a $\varphi $-fixed disc of $T$.
\end{example}

\begin{example}
\label{exm3} Let $X=%
%TCIMACRO{\U{211d} }%
%BeginExpansion
\mathbb{R}
%EndExpansion
$ be the usual metric space and consider the self-mapping $T:X\rightarrow X$
defined by%
\begin{equation*}
Tx=\left\{
\begin{array}{ccc}
2x & ; & x<-1 \\
x & ; & x\geq -1%
\end{array}%
\right.
\end{equation*}%
and the function $\varphi :$ $%
%TCIMACRO{\U{211d} }%
%BeginExpansion
\mathbb{R}
%EndExpansion
\rightarrow \left[ 0,\infty \right) $ defined by
\begin{equation*}
\varphi \left( x\right) =\left\{
\begin{array}{ccc}
0 & ; & x\geq -1 \\
\left\vert x\right\vert & ; & x<-1%
\end{array}%
\right. .
\end{equation*}%
We have%
\begin{eqnarray*}
\rho &=&\inf \left\{ d\left( Tx,x\right) :x\in X,x\neq Tx\right\} \\
&=&\inf \left\{ \left\vert 2x-x\right\vert :x<-1\right\} =\inf \left\{
\left\vert x\right\vert :x<-1\right\} =1
\end{eqnarray*}%
and%
\begin{eqnarray*}
\mu &=&\inf \left\{ \sqrt{d\left( Tx,x\right) }:x\in X,x\neq Tx\right\} \\
&=&\inf \left\{ \sqrt{\left\vert 2x-x\right\vert }:x<-1\right\} =\inf
\left\{ \sqrt{\left\vert x\right\vert }:x<-1\right\} \\
&=&1.
\end{eqnarray*}%
We show that $T$ is not a type $1$ $\varphi _{x_{0}}$-contraction $($resp.
generalized type $1$ $\varphi _{x_{0}}$-contraction$)$ with the point $%
x_{0}=0$. Indeed, for any $x<-1$, we have $d\left( x,Tx\right) >1$ and
\begin{equation}
\max \left\{ d\left( x,Tx\right) ,\varphi \left( Tx\right) ,\varphi \left(
x\right) \right\} =\max \left\{ \left\vert x\right\vert ,\left\vert
x\right\vert ,\left\vert x\right\vert \right\} =\left\vert x\right\vert ,
\label{eqn1}
\end{equation}%
\begin{equation}
\max \left\{ d\left( x,x_{0}\right) ,\varphi \left( x\right) ,\varphi \left(
x_{0}\right) \right\} =\max \left\{ \left\vert x\right\vert ,\left\vert
x\right\vert ,0\right\} =\left\vert x\right\vert ,  \label{eqn2}
\end{equation}%
\begin{equation}
\max \left\{ M\left( x,x_{0}\right) ,\varphi \left( x\right) ,\varphi \left(
x_{0}\right) \right\} =\max \left\{ \frac{3\left\vert x\right\vert ^{2}}{%
1+\left\vert x\right\vert },\left\vert x\right\vert ,0\right\} =\frac{%
3\left\vert x\right\vert ^{2}}{1+\left\vert x\right\vert }.  \label{eqn3}
\end{equation}%
Considering $($\ref{eqn1}$)$ and $($\ref{eqn2}$)$, we see that $T$ can not
be a type $1$ $\varphi _{x_{0}}$-contraction with the point $x_{0}=0$ for
any $k\in \left( 0,1\right) $. Also, considering $($\ref{eqn1}$)$ and $($\ref%
{eqn3}$)$, we see that $T$ can not be a generalized type $1$ $\varphi
_{x_{0}}$-contraction with the point $x_{0}=0$ for any $k\in \left( 0,\frac{1%
}{2}\right) $. Notice that the circle $C_{0,1}$ is a $\varphi $-fixed circle
of $T$ and the disc $D_{-1,1}$ is a $\varphi $-fixed disc of $T$.
\end{example}

\begin{remark}
Example \ref{exm3} shows that the converse statements of Theorem \ref{thm:21}
and Theorem \ref{thm:22} $($resp. Theorem \ref{thm:23} and Theorem \ref%
{thm:24}$)$ are not true everwhen.
\end{remark}

\subsection{\textbf{$\protect\varphi $-Fixed Circle (resp. $\protect\varphi $%
-Fixed Disc) Results via Type $2$ $\protect\varphi _{x_{0}}$-Contractions}}

Now, we define a new type of a $\varphi _{x_{0}}$-contraction.

\begin{definition}
\label{def:23} Let $(X,d)$ be a metric space, $T$ be a self-mapping of $X$
and $\varphi :X\rightarrow \left[ 0,\infty \right) $ be a given function. If
there exists a point $x_{0}\in X$ such that
\begin{equation}
d\left( Tx,x\right) >0\Rightarrow \max \left\{ d\left( x,Tx\right) ,\varphi
\left( Tx\right) \right\} +\varphi \left( x\right) \leq k\max \left\{
d\left( x,x_{0}\right) ,\varphi \left( x\right) \right\} +\varphi \left(
x_{0}\right) ,  \label{eq:26}
\end{equation}%
for all $x\in X$ and some $k\in \left( 0,1\right) $, then $T$ is called a
type $2$ $\varphi _{x_{0}}$-contraction.
\end{definition}

\begin{theorem}
\label{thm:25} Let $(X,d)$ be a metric space, the number $\rho $ be defined
as in $($\ref{the number rho}$)$ and $T:X\rightarrow X$ be a type $2$ $%
\varphi _{x_{0}}$-contraction with the point $x_{0}\in X$ and the given
function\ $\varphi :X\rightarrow \left[ 0,\infty \right) $. If $x_{0}\in
Z_{\varphi }$ and%
\begin{equation*}
\varphi \left( x\right) \leq d\left( Tx,x\right)
\end{equation*}%
for all $x\in C_{x_{0},\rho }$, then the circle $C_{x_{0},\rho }$ is a $%
\varphi $-fixed circle of $T$.
\end{theorem}

\begin{proof}
To show $x_{0}\in Fix\left( T\right) $, conversely, assume that $x_{0}\neq
Tx_{0}$. Then we have $d\left( Tx_{0},x_{0}\right) >0$ and using the
inequality (\ref{eq:26}) with the hypothesis $x_{0}\in Z_{\varphi }$, we find%
\begin{equation*}
\max \left\{ d\left( x_{0},Tx_{0}\right) ,\varphi \left( Tx_{0}\right)
\right\} +\varphi \left( x_{0}\right) \leq k\max \left\{ d\left(
x_{0},x_{0}\right) ,\varphi \left( x_{0}\right) \right\} +\varphi \left(
x_{0}\right) =0,
\end{equation*}%
and hence
\begin{equation*}
\max \left\{ d\left( x_{0},Tx_{0}\right) ,\varphi \left( Tx_{0}\right)
\right\} =0.
\end{equation*}%
This implies $d\left( x_{0},Tx_{0}\right) =0$, which is a contradiction with
our assumption. Therefore, it should be $Tx_{0}=x_{0}$, that is, $x_{0}\in
Fix\left( T\right) \cap Z_{\varphi }$.

Now we have two cases.

\textbf{Case 1.} If $\rho =0$, then clearly $C_{x_{0},\rho }=\left\{
x_{0}\right\} \subset Fix\left( T\right) \cap Z_{\varphi }$ and hence, the
circle $C_{x_{0},\rho }$ is a $\varphi $-fixed circle of $T$.

\textbf{Case 2.} Let $\rho >0$. For any $x\in C_{x_{0},\rho }$ with $Tx\neq
x $, we have
\begin{equation*}
\max \left\{ d\left( x,Tx\right) ,\varphi \left( Tx\right) \right\} +\varphi
\left( x\right) \leq k\max \left\{ d\left( x,x_{0}\right) ,\varphi \left(
x\right) \right\} .
\end{equation*}%
If $\max \left\{ d\left( x,x_{0}\right) ,\varphi \left( x\right) \right\}
=d\left( x,x_{0}\right) =\rho $, then by the definition of the number $\rho $%
, we get
\begin{equation*}
\max \left\{ d\left( x,Tx\right) ,\varphi \left( Tx\right) ,\varphi \left(
x\right) \right\} \leq k\rho \leq kd\left( x,Tx\right)
\end{equation*}%
and so $d\left( x,Tx\right) \leq kd\left( x,Tx\right) $, a contradiction by
the hypothesis $k\in \left( 0,1\right) $.

If $\max \left\{ d\left( x,x_{0}\right) ,\varphi \left( x\right) \right\}
=\varphi \left( x\right) $, we obtain%
\begin{equation*}
\max \left\{ d\left( x,Tx\right) ,\varphi \left( Tx\right) \right\} +\varphi
\left( x\right) \leq k\varphi \left( x\right)
\end{equation*}%
and hence $\varphi \left( x\right) \leq k\varphi \left( x\right) $, a
contradiction.

Then, it should be $Tx=x$, that is, $x\in Fix\left( T\right) $. By the
hypothesis $\varphi \left( x\right) \leq d\left( Tx,x\right) $, we have $%
\varphi \left( x\right) =0$ for all $x\in C_{x_{0},\rho }$. This implies $%
x\in Fix\left( T\right) \cap Z_{\varphi }$ for all $x\in C_{x_{0},\rho }$.
Consequently, we find $C_{x_{0},\rho }\subset Fix\left( T\right) \cap
Z_{\varphi }$ and hence, the circle $C_{x_{0},\rho }$ is a $\varphi $-fixed
circle of $T$.
\end{proof}

\begin{theorem}
\label{thm:26} Let $(X,d)$ be a metric space, the number $\rho $ be defined
as in $($\ref{the number rho}$)$ and $T:X\rightarrow X$ be a type $2$ $%
\varphi _{x_{0}}$-contraction with the point $x_{0}\in X$ and the given
function\ $\varphi :X\rightarrow \left[ 0,\infty \right) $. If $x_{0}\in
Z_{\varphi }$ and%
\begin{equation*}
\varphi \left( x\right) \leq d\left( Tx,x\right)
\end{equation*}%
for all $x\in D_{x_{0},\rho }$, then the circle $D_{x_{0},\rho }$ is a $%
\varphi $-fixed circle of $T$.
\end{theorem}

\begin{proof}
The proof is similar to the proof of Theorem \ref{thm:25}.
\end{proof}

\begin{definition}
\label{def:24} Let $(X,d)$ be a metric space, $T$ be a self-mapping of $X$
and $\varphi :X\rightarrow \left[ 0,\infty \right) $ be a given function. If
there exists a point $x_{0}\in X$ such that
\begin{equation}
d\left( Tx,x\right) >0\Rightarrow \max \left\{ d\left( x,Tx\right) ,\varphi
\left( Tx\right) \right\} +\varphi \left( x\right) \leq k\max \left\{
M\left( x,x_{0}\right) ,\varphi \left( x\right) \right\} +\varphi \left(
x_{0}\right) ,  \label{eq:27}
\end{equation}%
for all $x\in X$ and some $k\in \left( 0,\frac{1}{2}\right) $, then $T$ is
called a generalized type $2$ $\varphi _{x_{0}}$-contraction.
\end{definition}

\begin{theorem}
\label{thm:27} Let $(X,d)$ be a metric space, the number $\mu $ be defined
as in $($\ref{the number mu}$)$ and $T:X\rightarrow X$ be a generalized type
$2$ $\varphi _{x_{0}}$-contraction with the point $x_{0}\in X$ and the given
function\ $\varphi :X\rightarrow \left[ 0,\infty \right) $. If $x_{0}\in
Z_{\varphi }$ and the inequalities%
\begin{equation*}
d\left( Tx,x_{0}\right) \leq \mu \text{,}
\end{equation*}%
\begin{equation*}
\varphi \left( x\right) \leq d\left( Tx,x\right)
\end{equation*}%
hold for all $x\in C_{x_{0},\mu }$, then the circle $C_{x_{0},\mu }$ is a $%
\varphi $-fixed circle of $T$.
\end{theorem}

\begin{proof}
Suppose that $d\left( Tx_{0},x_{0}\right) >0$. Using the inequality (\ref%
{eq:27}) and the hypothesis $x_{0}\in Z_{\varphi }$, we find%
\begin{eqnarray*}
\max \left\{ d\left( x_{0},Tx_{0}\right) ,\varphi \left( Tx_{0}\right)
\right\} +\varphi \left( x_{0}\right) &\leq &k\max \left\{ M\left(
x_{0},x_{0}\right) ,\varphi \left( x_{0}\right) \right\} +\varphi \left(
x_{0}\right) \\
&\leq &kM\left( x_{0},x_{0}\right) =kd(x_{0},Tx_{0})
\end{eqnarray*}%
and hence
\begin{equation*}
\max \left\{ d\left( x_{0},Tx_{0}\right) ,\varphi \left( Tx_{0}\right)
\right\} \leq kd(x_{0},Tx_{0}).
\end{equation*}%
This implies $d\left( x_{0},Tx_{0}\right) \leq kd(x_{0},Tx_{0})$, which is a
contradiction since $k\in \left( 0,\frac{1}{2}\right) $. It should be $%
Tx_{0}=x_{0}$, that is, we get $x_{0}\in Fix\left( T\right) $. Therefore, we
have $x_{0}\in Fix\left( T\right) \cap Z_{\varphi }$.

If $\mu =0$, then clearly $C_{x_{0},\mu }=\left\{ x_{0}\right\} \subset
Fix\left( T\right) \cap Z_{\varphi }$ and hence, the circle $C_{x_{0},\mu }$
is a $\varphi $-fixed circle of $T$.

Now, let $\mu >0$. For any $x\in C_{x_{0},\mu }$ with $Tx\neq x$, we have
\begin{equation*}
\max \left\{ d\left( x,Tx\right) ,\varphi \left( Tx\right) \right\} +\varphi
\left( x\right) \leq k\max \left\{ M\left( x,x_{0}\right) ,\varphi \left(
x\right) \right\} .
\end{equation*}%
If $\max \left\{ M\left( x,x_{0}\right) ,\varphi \left( x\right) \right\}
=\varphi \left( x\right) $, we obtain%
\begin{equation*}
\max \left\{ d\left( x,Tx\right) ,\varphi \left( Tx\right) \right\} +\varphi
\left( x\right) \leq k\varphi \left( x\right)
\end{equation*}%
and hence $\varphi \left( x\right) \leq k\varphi \left( x\right) $, a
contradiction by the hypothesis $k\in \left( 0,\frac{1}{2}\right) $.

Let $\max \left\{ M\left( x,x_{0}\right) ,\varphi \left( x\right) \right\}
=M\left( x,x_{0}\right) $. We have
\begin{eqnarray*}
M(x,x_{0}) &=&\max \left\{ d(x,x_{0}),d(x,Tx),d(x_{0},Tx_{0}),\left[ \frac{%
d(x,Tx_{0})+d(x_{0},Tx)}{1+d(x,Tx)+d(x_{0},Tx_{0})}\right] d(x,x_{0})\right\}
\\
&\leq &\max \left\{ \mu ,d(x,Tx),0,\left[ \frac{\mu +\mu }{1+d(x,Tx)}\right]
\mu \right\} \\
&=&\max \left\{ \mu ,d(x,Tx),0,\frac{2\mu ^{2}}{1+d(x,Tx)}\right\} .
\end{eqnarray*}%
If $M(x,x_{0})=d(x,Tx)$ then we get
\begin{equation*}
\max \left\{ d\left( x,Tx\right) ,\varphi \left( Tx\right) \right\} +\varphi
\left( x\right) \leq kd\left( x,Tx\right)
\end{equation*}%
and so $d\left( x,Tx\right) \leq kd\left( x,Tx\right) $, a contradiction by
the hypothesis $k\in \left( 0,\frac{1}{2}\right) $. If $M(x,x_{0})=\frac{%
2\mu ^{2}}{1+d(x,Tx)}$ then by the definition of the number $\mu $ we have%
\begin{eqnarray*}
\max \left\{ d\left( x,Tx\right) ,\varphi \left( Tx\right) \right\} +\varphi
\left( x\right) &\leq &k\frac{2\mu ^{2}}{1+d(x,Tx)} \\
&\leq &\frac{2k\left( \sqrt{d(x,Tx)}\right) ^{2}}{1+d(x,Tx)} \\
&<&2k(x,Tx)
\end{eqnarray*}%
and hence $d\left( x,Tx\right) <2k(x,Tx)$, a contradiction by the hypothesis
$k\in \left( 0,\frac{1}{2}\right) $. Then, it should be $Tx=x$, that is, $%
x\in Fix\left( T\right) $ in all of the above cases. By the hypothesis $%
\varphi \left( x\right) \leq d\left( Tx,x\right) $, we have $\varphi \left(
x\right) =0$ for all $x\in C_{x_{0},\mu }$. This implies $x\in Fix\left(
T\right) \cap Z_{\varphi }$ for all $x\in C_{x_{0},\mu }$. Consequently, we
find $C_{x_{0},\mu }\subset Fix\left( T\right) \cap Z_{\varphi }$ and hence,
the circle $C_{x_{0},\mu }$ is a $\varphi $-fixed circle of $T$.
\end{proof}

\begin{theorem}
\label{thm:28} Let $(X,d)$ be a metric space, the number $\mu $ be defined
as in $($\ref{the number mu}$)$ and $T:X\rightarrow X$ be a generalized type
$2$ $\varphi _{x_{0}}$-contraction with the point $x_{0}\in X$ and the given
function\ $\varphi :X\rightarrow \left[ 0,\infty \right) $. If $x_{0}\in
Z_{\varphi }$ and the inequalities%
\begin{equation*}
d\left( Tx,x_{0}\right) \leq \mu \text{,}
\end{equation*}%
\begin{equation*}
\varphi \left( x\right) \leq d\left( Tx,x\right)
\end{equation*}%
hold for all $x\in D_{x_{0},\mu }$, then the disc $D_{x_{0},\mu }$ is a $%
\varphi $-fixed disc of $T$.
\end{theorem}

\begin{proof}
The proof is similar to the proof of Theorem \ref{thm:27}.
\end{proof}

\subsection{\textbf{$\protect\varphi $-Fixed Circle (resp. $\protect\varphi $%
-Fixed Disc) Results via Type $3$ $\protect\varphi _{x_{0}}$-Contractions}}

Now, we define a new type of a $\varphi _{x_{0}}$-contraction.

\begin{definition}
\label{def:25} Let $(X,d)$ be a metric space, $T$ be a self-mapping of $X$
and $\varphi :X\rightarrow \left[ 0,\infty \right) $ be a given function. If
there exists a point $x_{0}\in X$ such that
\begin{equation}
d\left( Tx,x\right) >0\Rightarrow d\left( x,Tx\right) +\varphi \left(
Tx\right) +\varphi \left( x\right) \leq k\left[ d\left( x,x_{0}\right)
+\varphi \left( x\right) +\varphi \left( x_{0}\right) \right] ,
\label{eq:28}
\end{equation}%
for all $x\in X$ and some $k\in \left( 0,1\right) $, then $T$ is called a
type $3$ $\varphi _{x_{0}}$-contraction.
\end{definition}

\begin{theorem}
\label{thm:29} Let $(X,d)$ be a metric space, the number $\rho $ be defined
as in $($\ref{the number rho}$)$ and $T:X\rightarrow X$ be a type $3$ $%
\varphi _{x_{0}}$-contraction with the point $x_{0}\in X$ and the given
function\ $\varphi :X\rightarrow \left[ 0,\infty \right) $. If $x_{0}\in
Z_{\varphi }$ and%
\begin{equation*}
\varphi \left( x\right) \leq d\left( Tx,x\right)
\end{equation*}%
for all $x\in C_{x_{0},\rho }$, then the circle $C_{x_{0},\rho }$ is a $%
\varphi $-fixed circle of $T$.
\end{theorem}

\begin{proof}
If $d\left( Tx_{0},x_{0}\right) >0$, using the inequality (\ref{eq:28}) and
the hypothesis $x_{0}\in Z_{\varphi }$, we find%
\begin{equation*}
d\left( x_{0},Tx_{0}\right) +\varphi \left( Tx_{0}\right) +\varphi \left(
x_{0}\right) \leq k\left[ d\left( x_{0},x_{0}\right) +\varphi \left(
x_{0}\right) +\varphi \left( x_{0}\right) \right] =0,
\end{equation*}%
and hence
\begin{equation*}
d\left( x_{0},Tx_{0}\right) +\varphi \left( Tx_{0}\right) =0.
\end{equation*}%
This implies $d\left( x_{0},Tx_{0}\right) =0$, which is a contradiction with
our assumption. Therefore, it should be $Tx_{0}=x_{0}$, that is, $x_{0}\in
Fix\left( T\right) \cap Z_{\varphi }$.

Now we have two cases.

\textbf{Case 1.} If $\rho =0$, then clearly $C_{x_{0},\rho }=\left\{
x_{0}\right\} \subset Fix\left( T\right) \cap Z_{\varphi }$ and hence, the
circle $C_{x_{0},\rho }$ is a $\varphi $-fixed circle of $T$.

\textbf{Case 2.} Let $\rho >0$. For any $x\in C_{x_{0},\rho }$ with $Tx\neq
x $, we have
\begin{eqnarray*}
d\left( x,Tx\right) +\varphi \left( Tx\right) +\varphi \left( x\right) &\leq
&k\left[ d\left( x,x_{0}\right) +\varphi \left( x\right) \right] \\
&<&d\left( x,x_{0}\right) +\varphi \left( x\right)
\end{eqnarray*}%
and so,%
\begin{equation*}
d\left( x,Tx\right) +\varphi \left( Tx\right) <d\left( x,x_{0}\right) =\rho .
\end{equation*}%
This last inequality implies $d\left( x,Tx\right) <\rho $, which is a
contradiction by the definition of $\rho .$ Then, it should be $Tx=x$, that
is, $x\in Fix\left( T\right) $. Since we have $\varphi \left( x\right) \leq
d\left( Tx,x\right) $ for all $x\in C_{x_{0},\rho }$, we obtain $\varphi
\left( x\right) =0$ and hence, $x\in Fix\left( T\right) \cap Z_{\varphi }$
for all $x\in C_{x_{0},\rho }$. Consequently, we get $C_{x_{0},\rho }\subset
Fix\left( T\right) \cap Z_{\varphi }$, that is, the circle $C_{x_{0},\rho }$
is a $\varphi $-fixed circle of $T$.
\end{proof}

\begin{theorem}
\label{thm:210} Let $(X,d)$ be a metric space, the number $\rho $ be defined
as in $($\ref{the number rho}$)$ and $T:X\rightarrow X$ be a type $3$ $%
\varphi _{x_{0}}$-contraction with the point $x_{0}\in X$ and the given
function\ $\varphi :X\rightarrow \left[ 0,\infty \right) $. If $x_{0}\in
Z_{\varphi }$ and%
\begin{equation*}
\varphi \left( x\right) \leq d\left( Tx,x\right)
\end{equation*}%
for all $x\in D_{x_{0},\rho }$, then the disc $D_{x_{0},\rho }$ is a $%
\varphi $-fixed disc of $T$.
\end{theorem}

\begin{proof}
The proof is similar to the proof of Theorem \ref{thm:27}.
\end{proof}

\begin{definition}
\label{def:26} Let $(X,d)$ be a metric space, $T$ be a self-mapping of $X$
and $\varphi :X\rightarrow \left[ 0,\infty \right) $ be a given function. If
there exists a point $x_{0}\in X$ such that
\begin{equation}
d\left( Tx,x\right) >0\Rightarrow d\left( x,Tx\right) +\varphi \left(
Tx\right) +\varphi \left( x\right) \leq k\left[ M\left( x,x_{0}\right)
+\varphi \left( x\right) +\varphi \left( x_{0}\right) \right] ,
\label{eq:29}
\end{equation}%
for all $x\in X$ and some $k\in \left( 0,\frac{1}{2}\right) $, then $T$ is
called a generalized type $3$ $\varphi _{x_{0}}$-contraction.
\end{definition}

\begin{theorem}
\label{thm:211} Let $(X,d)$ be a metric space, the number $\mu $ be defined
as in $($\ref{the number mu}$)$ and $T:X\rightarrow X$ be a generalized type
$3$ $\varphi _{x_{0}}$-contraction with the point $x_{0}\in X$ and the given
function\ $\varphi :X\rightarrow \left[ 0,\infty \right) $. If $x_{0}\in
Z_{\varphi }$ and the inequalities%
\begin{equation*}
d\left( Tx,x_{0}\right) \leq \mu \text{,}
\end{equation*}%
\begin{equation*}
\varphi \left( x\right) \leq d\left( Tx,x\right)
\end{equation*}%
hold for all $x\in C_{x_{0},\mu }$, then the circle $C_{x_{0},\mu }$ is a $%
\varphi $-fixed circle of $T$.
\end{theorem}

\begin{proof}
If $d\left( Tx_{0},x_{0}\right) >0$, using the inequality (\ref{eq:29}) and
the hypothesis $x_{0}\in Z_{\varphi }$, we find%
\begin{eqnarray*}
d\left( x_{0},Tx_{0}\right) +\varphi \left( Tx_{0}\right) +\varphi \left(
x_{0}\right) &\leq &k\left[ M\left( x_{0},x_{0}\right) +\varphi \left(
x_{0}\right) +\varphi \left( x_{0}\right) \right] \\
&\leq &kM\left( x_{0},x_{0}\right) =kd(x_{0},Tx_{0})
\end{eqnarray*}%
and hence
\begin{equation*}
d\left( x_{0},Tx_{0}\right) +\varphi \left( Tx_{0}\right) \leq
kd(x_{0},Tx_{0}).
\end{equation*}%
This implies $d\left( x_{0},Tx_{0}\right) \leq kd(x_{0},Tx_{0})$, which is a
contradiction since $k\in \left( 0,\frac{1}{2}\right) $. Then, it should be $%
Tx_{0}=x_{0}$, that is, we get $x_{0}\in Fix\left( T\right) $. Therefore, we
have $x_{0}\in Fix\left( T\right) \cap Z_{\varphi }$.

If $\mu =0$, then clearly $C_{x_{0},\mu }=\left\{ x_{0}\right\} \subset
Fix\left( T\right) \cap Z_{\varphi }$ and hence, the circle $C_{x_{0},\mu }$
is a $\varphi $-fixed circle of $T$.

Now, let $\mu >0$. For any $x\in C_{x_{0},\mu }$ with $Tx\neq x$, we have
\begin{eqnarray*}
d\left( x,Tx\right) +\varphi \left( Tx\right) +\varphi \left( x\right) &\leq
&k\left[ M\left( x,x_{0}\right) +\varphi \left( x\right) \right] \\
&<&M\left( x,x_{0}\right) +\varphi \left( x\right)
\end{eqnarray*}%
and so,%
\begin{equation*}
d\left( x,Tx\right) +\varphi \left( Tx\right) <M\left( x,x_{0}\right) .
\end{equation*}%
By the inequality $d\left( Tx,x_{0}\right) \leq \mu $, we have
\begin{equation*}
M(x,x_{0})\leq \max \left\{ \mu ,d(x,Tx),0,\frac{2\mu ^{2}}{1+d(x,Tx)}%
\right\} .
\end{equation*}%
If $M(x,x_{0})=d(x,Tx)$, then we get
\begin{equation*}
d\left( x,Tx\right) +\varphi \left( Tx\right) \leq kd\left( x,Tx\right)
\end{equation*}%
and so $d\left( x,Tx\right) \leq kd\left( x,Tx\right) $, a contradiction by
the hypothesis $k\in \left( 0,\frac{1}{2}\right) $. If $M(x,x_{0})=\frac{%
2\mu ^{2}}{1+d(x,Tx)}$ then by the definition of the number $\mu $ we have%
\begin{eqnarray*}
\max \left\{ d\left( x,Tx\right) ,\varphi \left( Tx\right) \right\} +\varphi
\left( x\right) &\leq &k\frac{2\mu ^{2}}{1+d(x,Tx)} \\
&\leq &\frac{2k\left( \sqrt{d(x,Tx)}\right) ^{2}}{1+d(x,Tx)} \\
&<&2k(x,Tx)
\end{eqnarray*}%
and hence $d\left( x,Tx\right) <2k(x,Tx)$, a contradiction by the hypothesis
$k\in \left( 0,\frac{1}{2}\right) $. Then, it should be $Tx=x$, that is, $%
x\in Fix\left( T\right) $ in all of the above cases. By the hypothesis $%
\varphi \left( x\right) \leq d\left( Tx,x\right) $, we have $\varphi \left(
x\right) =0$ for all $x\in C_{x_{0},\mu }$. This implies $x\in Fix\left(
T\right) \cap Z_{\varphi }$ for all $x\in C_{x_{0},\mu }$. Consequently, we
find $C_{x_{0},\mu }\subset Fix\left( T\right) \cap Z_{\varphi }$ and hence,
the circle $C_{x_{0},\mu }$ is a $\varphi $-fixed circle of $T$.
\end{proof}

\begin{theorem}
\label{thm:212} Let $(X,d)$ be a metric space, the number $\mu $ be defined
as in $($\ref{the number mu}$)$ and $T:X\rightarrow X$ be a generalized type
$3$ $\varphi _{x_{0}}$-contraction with the point $x_{0}\in X$ and the given
function\ $\varphi :X\rightarrow \left[ 0,\infty \right) $. If $x_{0}\in
Z_{\varphi }$ and the inequalities%
\begin{equation*}
d\left( Tx,x_{0}\right) \leq \mu \text{,}
\end{equation*}%
\begin{equation*}
\varphi \left( x\right) \leq d\left( Tx,x\right)
\end{equation*}%
hold for all $x\in D_{x_{0},\mu }$, then the disc $D_{x_{0},\mu }$ is a $%
\varphi $-fixed disc of $T$.
\end{theorem}

\begin{proof}
The proof is similar to the proof of Theorem \ref{thm:211}.
\end{proof}

Finally, we give an example of a self-mapping $T$ such that $T$ satisfies
all conditions of Theorem \ref{thm:21}, Theorem \ref{thm:23}, Theorem \ref%
{thm:25}, Theorem \ref{thm:27}, Theorem \ref{thm:29} and Theorem \ref%
{thm:211}.

\begin{example}
Let $X=\left\{ -6,-4,-2,0,1,2,4,5\right\} \cup \left[ 6,\infty \right) $
with the usual metric $d\left( x,y\right) =\left\vert x-y\right\vert $ and
consider the self-mapping $T:X\rightarrow X$ defined by%
\begin{equation*}
Tx=\left\{
\begin{array}{ccc}
x & , & x\neq 5 \\
1 & ; & x=5%
\end{array}%
\right.
\end{equation*}%
and the function $\varphi :$ $X\rightarrow \left[ 0,\infty \right) $ defined
by
\begin{equation*}
\varphi \left( x\right) =\left\{
\begin{array}{ccc}
x^{5}-20x^{3}+64x & ; & x\in X-\left\{ 1,5\right\} \\
0 & ; & x\in \left\{ -6,1,5\right\}%
\end{array}%
\right. .
\end{equation*}%
We have%
\begin{eqnarray*}
\rho &=&\inf \left\{ d\left( Tx,x\right) :x\in X,x\neq Tx\right\} \\
&=&\left\vert 5-1\right\vert =4
\end{eqnarray*}%
and
\begin{eqnarray*}
\mu &=&\inf \left\{ \sqrt{d\left( Tx,x\right) }:x\in X,x\neq Tx\right\} \\
&=&\sqrt{4}=2.
\end{eqnarray*}%
Observe that $Fix\left( T\right) =X-\left\{ 5\right\} $, $Z_{\varphi
}=\left\{ -6,-4,-2,0,1,2,4,5\right\} $ and $Fix\left( T\right) \cap
Z_{\varphi }=\left\{ -6,-4,-2,0,1,2,4\right\} $.

Now, we show that $T$ is a type $1$ $\varphi _{x_{0}}$-contraction with the
point $x_{0}=0$ and $k=\frac{9}{10}$. Indeed, we have
\begin{equation*}
\max \left\{ \left\vert 5-1\right\vert ,0,0\right\} =4\leq \frac{9}{10}\max
\left\{ \left\vert 5-0\right\vert ,0,0\right\} =\frac{45}{10}
\end{equation*}%
for $x=5$. Clearly, conditions of Theorem \ref{thm:21} are satisfied by $T$.
We get
\begin{equation*}
C_{0,4}=\left\{ -4,4\right\} \subset Fix\left( T\right) \cap Z_{\varphi }.
\end{equation*}%
Hence, the circle $C_{0,4}$ is a $\varphi $-fixed circle of $T$.

$T$ is a generalized type $1$ $\varphi _{x_{0}}$-contraction with the point $%
x_{0}=-4$ and $k=\frac{1}{4}$. Indeed, we have%
\begin{eqnarray*}
M(5,4) &=&\max \left\{ d(5,-4),d(5,1),d(-4,-4),\left[ \frac{d(5,-4)+d(-4,1)}{%
1+d(5,1)+d(-4,-4)}\right] d(5,-4)\right\} \\
&=&\max \left\{ 9,4,0,\frac{126}{5}\right\} =\frac{126}{5}
\end{eqnarray*}%
and so,%
\begin{equation*}
\max \left\{ \left\vert 5-1\right\vert ,0,0\right\} =4\leq \frac{1}{4}\max
\left\{ \frac{126}{5},0,0\right\} =\frac{63}{10}.
\end{equation*}

Clearly, $T$ also satisfies the conditions of Theorem \ref{thm:23} and the
circle $C_{-4,2}=\left\{ -6,-2\right\} $ is another $\varphi $-fixed circle
of $T$.

Similarly, it is easy to verify that $T$ is a type $2$ $\varphi _{x_{0}}$%
-contraction with the point $x_{0}=0$ and $k=\frac{9}{10};$ a generalized
type $2$ $\varphi _{x_{0}}$-contraction with the point $x_{0}=-4$ and $k=%
\frac{1}{4};$ a type $3$ $\varphi _{x_{0}}$-contraction with the point $%
x_{0}=0$ and $k=\frac{9}{10};$ a generalized type $2$ $\varphi _{x_{0}}$%
-contraction with the point $x_{0}=-4$ and $k=\frac{1}{4}$. The discs $%
D_{0,4}=\left\{ -4,-2,0,1,2,4\right\} $ and $D_{-4,2}=\left\{
-6,-4,-2\right\} $ are $\varphi $-fixed discs of $T$.
\end{example}

\section{\textbf{Conclusions and Future Work}}

We have given some solutions to a recent open problem related to the
geometric properties of $\varphi $-fixed points of self-mappings of a metric
space in the non-unique fixed point cases. We have seen that any zero of a
given function will produce a fixed circle (resp. fixed disc) contained in
the fixed point set of a self-mapping on a metric space. Considering the
examples given in the paper, we have deduce that a $\varphi $-fixed circle
(resp. $\varphi $-fixed disc) need not to be unique. Hence, the
investigation of the uniqueness condition(s) of a $\varphi $-fixed circle
(resp. $\varphi $-fixed disc) can be considered as a future scope of this
paper.

On the other hand, theoretical fixed point results and examples of
self-mappings are important tools in the study of neural networks. The
typical form of a partitioned real valued activation function is the
following:%
\begin{equation*}
f(x)=\left\{
\begin{array}{ccc}
f_{0}(x) & ; & x<0 \\
f_{1}(x) & ; & x\geq 0%
\end{array}%
\right. ,
\end{equation*}%
where $f_{0}(x)$ and $f_{2}(x)$ are local functions for positive and
negative regions (see \cite{Lee} for more details). Many activation
functions such as LReLU and SELU are in this form. The typical form of more
complicated cases is the following:%
\begin{equation}
f(x)=\left\{
\begin{array}{ccc}
f_{0}(x) & ; & x<x_{0} \\
f_{1}(x) & ; & x_{0}<x\leq x_{1} \\
\cdots &  &  \\
f_{n-1}(x) & ; & x_{n-2}<x\leq x_{n-1} \\
f_{n}(x) & ; & x_{n-1}<x%
\end{array}%
\right. .  \label{general}
\end{equation}

Now, we propose two examples of real and complex valued activation functions
of which fixed point sets contain a $\varphi $-fixed circle.

\begin{example}
Let $X=%
%TCIMACRO{\U{211d} }%
%BeginExpansion
\mathbb{R}
%EndExpansion
$ with the usual metric. We define an example of a real valued activation
function using the general form $($\ref{general}$)$ by%
\begin{equation*}
Tx=\left\{
\begin{array}{ccc}
\left\vert x\right\vert +1 & ; & x<0 \\
x & ; & 0<x\leq 4 \\
x+1 & ; & 4<x\leq 8 \\
x+2 & ; & x>8%
\end{array}%
\right. .
\end{equation*}%
Notice that we have
\begin{equation*}
Fix\left( T\right) =\left( 0,4\right] \text{ and }\rho =\inf \left\{ d\left(
Tx,x\right) :x\in
%TCIMACRO{\U{211d} }%
%BeginExpansion
\mathbb{R}
%EndExpansion
,x\neq Tx\right\} =1.
\end{equation*}%
If we define the function $\varphi :%
%TCIMACRO{\U{211d} }%
%BeginExpansion
\mathbb{R}
%EndExpansion
\rightarrow \left[ 0,\infty \right) $ by%
\begin{equation*}
\varphi \left( x\right) =\left\{
\begin{array}{ccc}
\left\vert x\right\vert & ; & x<0 \\
0 & ; & 0<x\leq 3 \\
8-x & ; & 3<x\leq 8 \\
x & ; & x>8%
\end{array}%
\right. ,
\end{equation*}%
we get
\begin{equation*}
Z_{\varphi }=\left( 0,3\right] \cup \left\{ 8\right\} \text{ and }Fix\left(
T\right) \cap Z_{\varphi }=\left( 0,3\right] .\text{ }
\end{equation*}%
The fixed points of $T$ belonging in the interval $\left( 0,3\right] $ are
special because of the reason that they are also zeros of the function $%
\varphi $. The circle $C_{2,1}=\left\{ 1,3\right\} $ $($resp. the disc $%
D_{2,1}=\left[ 1,3\right] )$ contained in the set $Fix\left( T\right) \cap
Z_{\varphi }$ is a $\varphi $-fixed circle $($resp. $\varphi $-fixed disc$)$
of $T$.
\end{example}

This approach can also be used to define complex activation functions as we
have seen in the following example.

\begin{example}
For example, let $X=%
%TCIMACRO{\U{2102} }%
%BeginExpansion
\mathbb{C}
%EndExpansion
$ with the usual metric and define the self-mapping $S$ by%
\begin{equation*}
Sz=\left\{
\begin{array}{ccc}
-\overline{z} & ; & x<-4 \\
z+1 & ; & -4<x\leq -2 \\
z & ; & -2<x\leq 1 \\
z-1 & ; & x>1%
\end{array}%
\right. .
\end{equation*}%
We have
\begin{equation*}
Fix\left( S\right) =\left\{ z=x+iy\in
%TCIMACRO{\U{2102} }%
%BeginExpansion
\mathbb{C}
%EndExpansion
:-2<x\leq 1\right\} \text{ }
\end{equation*}%
and%
\begin{equation*}
\text{ }\rho =\inf \left\{ d\left( Sz,z\right) :z\in
%TCIMACRO{\U{2102} }%
%BeginExpansion
\mathbb{C}
%EndExpansion
,z\neq Sz\right\} =1.
\end{equation*}%
Define the function $\varphi :%
%TCIMACRO{\U{2102} }%
%BeginExpansion
\mathbb{C}
%EndExpansion
\rightarrow \left[ 0,\infty \right) $ by%
\begin{equation*}
\varphi \left( z\right) =\left\{
\begin{array}{ccc}
\left\vert z\right\vert -1 & ; & \left\vert z\right\vert \geq 1 \\
\left\vert z\right\vert & ; & \left\vert z\right\vert <1%
\end{array}%
\right. .
\end{equation*}%
We find%
\begin{equation*}
Z_{\varphi }=\left\{ z\in
%TCIMACRO{\U{2102} }%
%BeginExpansion
\mathbb{C}
%EndExpansion
:\left\vert z\right\vert =1\right\} \cup \left\{ 0\right\} \text{ and }%
Fix\left( T\right) \cap Z_{\varphi }=\left\{ z\in
%TCIMACRO{\U{2102} }%
%BeginExpansion
\mathbb{C}
%EndExpansion
:\left\vert z\right\vert =1\right\} \cup \left\{ 0\right\} .
\end{equation*}%
The circle $C_{0,1}=\left\{ z\in
%TCIMACRO{\U{2102} }%
%BeginExpansion
\mathbb{C}
%EndExpansion
:\left\vert z\right\vert =1\right\} $ $($resp. the disc $D_{0,1}=\left\{
z\in
%TCIMACRO{\U{2102} }%
%BeginExpansion
\mathbb{C}
%EndExpansion
:\left\vert z\right\vert \leq 1\right\} )$ is a $\varphi $-fixed circle $($%
resp. $\varphi $-fixed disc$)$ of $S$.
\end{example}

Thus, these kind general activation functions can be used in the study of
neural networks to obtain more properties with a geometric approach.

\textbf{Acknowledgement.} %The authors would like to thank the referees for
%their positive and insightful comments on the manuscript.
This work is supported by the Scientific Research Projects Unit of Bal\i
kesir University under the project number 2020/019.


\begin{thebibliography}{99}
\bibitem{Agarwal} Agarwal, P., Jleli, M., Samet, B.: On fixed points that
belong to the zero set of a certain function. In Fixed Point Theory in
Metric Spaces (pp. 101-122). Springer, Singapore (2018).

\bibitem{Ali} Ali, M.U., Muangchoo-in, K., Kumam, P.: Zeroes and Fixed
Points of Different Functions via Contraction Type Conditions, In:
International Econometric Conference of Vietnam (pp. 353-359). Springer,
Cham.

\bibitem{Gopal} Gopal, D., Budhia, L. M., Jain, S.: A relation theoretic
approach for $\varphi $-fixed point result in metric space with an
application to an integral equation. Communications in Nonlinear Analysis,
6(1) (2019), 89-95.

\bibitem{Imdad} Imdad, M., Khan, A. R., Saleh, H. N., Alfaqih, W. M.: Some $%
\varphi $-fixed point results for $(F,\varphi ,\alpha -\psi )$-contractive
type mappings with applications. Mathematics, 7(2) (2019), 122.

\bibitem{Jleli} Jleli, M., Samet, B., Vetro, C.: Fixed point theory in
partial metric spaces via $\varphi $-fixed point's concept in metric spaces.
Journal of Inequalities and Applications, \textbf{2014} (1) (2014), 1-9.

\bibitem{Karapinar} Karapinar, E., O'Regan, D., Samet, B.: On the existence
of fixed points that belong to the zero set of a certain function. Fixed
Point Theory Appl. 2015, 2015:152, 14 pp.

\bibitem{Kumrod} Kumrod, P., Sintunavarat, W.: A new contractive condition
approach to $\varphi $-fixed point results in metric spaces and its
applications. J. Comput. Appl. Math. 311 (2017), 194-204.

\bibitem{Kumrod 2019} Kumrod, P., Sintunavarat, W.: On new fixed point
results in various distance spaces via $\varphi $-fixed point theorems in $D$%
-generalized metric spaces with numerical results. Journal of Fixed Point
Theory and Applications, 21(3) (2019), 1-14.

\bibitem{Lee} Lee, H., Park, H. S.: A generalization method of partitioned
activation function for complex number. arXiv preprint arXiv:1802.02987,
(2018).

\bibitem{Mlaiki arxiv} Mlaiki, N., \"{O}zg\"{u}r, N., Ta\c{s}, N.: New
fixed-circle results related to $F_{c}$-contractive and $F_{c}$-expanding
mappings on metric spaces. arXiv:2101.10770.

\bibitem{Ozgur-malaysian} \"{O}zg\"{u}r, N.Y., Ta\c{s}, N.: Some
fixed-circle theorems on metric spaces. Bull. Malays. Math. Sci. Soc.
\textbf{42} (4) (2019), 1433-1449.

\bibitem{Ozgur-Aip} \"{O}zg\"{u}r, N.Y., Ta\c{s}, N.: Some fixed-circle
theorems and discontinuity at fixed circle. AIP Conference Proceedings 1926,
020048 (2018).

\bibitem{Ozgur-chapter} \"{O}zg\"{u}r, N.Y., Ta\c{s}, N.: Generalizations of
metric spaces: from the fixed-point theory to the fixed-circle theory, In:
Rassias T. (eds) Applications of Nonlinear Analysis. Springer Optimization
and Its Applications, vol 134. Springer, Cham (2018).

\bibitem{Ozgur-fixed disc} \"{O}zg\"{u}r, N.: Fixed-disc results via
simulation functions, Turk. J. Math. \textbf{43} (6) (2019), 2794-2805.

\bibitem{Ozgur-simulation} \"{O}zg\"{u}r, N., Ta\c{s}, N.: Geometric
properties of fixed points and simulation functions, arXiv:2102.05417

\bibitem{Ozgur 2021} \"{O}zg\"{u}r, N., Ta\c{s}, N.: New discontinuity
results at fixed point on metric spaces. J. Fixed Point Theory Appl. \textbf{%
23} (2) (2021) 28.

\bibitem{Tas 2020} Ta\c{s}, N.: Bilateral-type solutions to the fixed-circle
problem with rectified linear units application. Turkish J. Math. 44 (4)
(2020), 1330-1344.

\bibitem{Vetro} Vetro, F.: Fixed points that are zeros of a given function.
Advances in Metric Fixed Point Theory and Applications (2021), 157.
\end{thebibliography}
\end{document}